\documentclass[12pt]{article}

\usepackage{tikz}
\usepackage{graphicx}
\usepackage{verbatim}
\usepackage{array,fullpage,multirow}
\usepackage{latexsym}
\usepackage{enumitem}
\usepackage{amssymb,amsfonts,amsmath,amsthm}
\usepackage{url}
\usepackage[colorlinks=true,linkcolor=blue,citecolor=red]{hyperref}

\newcommand{\eqdef}{\stackrel{\rm{def}}{=}}
\newcommand{\emptyslot}{\sqcup}
\newcommand{\bb}[1][n]{\mathcal{B}_{#1}}
\newcommand{\cc}[1][n]{\{0,1\}^{#1}}
\newcommand{\markedcube}[1][n]{\{0,1,\hat{0}, \hat{1}\}^{#1}}
\newcommand{\regcube}[1][n]{\{0,1,\emptyslot\}^{#1}}

\newcommand{\floor}[1]{{\lfloor#1\rfloor}}
\newcommand{\ceil}[1]{{\lceil#1\rceil}}

\newcommand{\dist}{{\sf distance}}

\newcommand{\maj}{{\sf majority}}
\newcommand{\parity}{{\sf parity}}
\newcommand{\even}{{\sf even}}

\newcommand{\maxstretch}{{\sf maxStretch}}
\newcommand{\avgstretch}{{\sf avgStretch}}

\newcommand{\poly}{{\rm poly}}

\renewcommand{\P}{\ensuremath{\mathcal{P}}}
\newcommand{\N}{{\mathbb N}}
\newcommand{\R}{{\mathbb R}}

\newcommand{\eps}{\varepsilon}
\newcommand{\seq}{\subseteq}

\newcommand{\half}{\frac{1}{2}}
\newcommand{\AC}[1]{\mathbf{AC^{#1}}}
\newcommand{\NC}[1]{\mathbf{NC^{#1}}}
\newcommand{\TC}[1]{\mathbf{TC^{#1}}}
\newcommand{\und}{\underline}
\newcommand{\markhat}{{\sf mark}}
\renewcommand{\int}{{\sf int}}

\newcommand{\D}{{\mathcal D}}
\newcommand{\etal}{et~al.}

\def\OR{\ensuremath{\mathsf{OR}}}

\newcommand{\f}{\psi}
\newcommand{\h}{\phi}

\newcommand{\Inf}{\mathbf{Inf}}
\newcommand{\Ex}{\mathop{{\bf E}\/}}
\renewcommand{\Pr}{\mathop{\bf Pr\/}}

\newtheorem{theorem}{Theorem}
\newtheorem{corollary}{Corollary}[section]
\newtheorem{proposition}[corollary]{Proposition}
\newtheorem{definition}[corollary]{Definition}

\newtheorem{claim}[corollary]{Claim}

\newtheorem{remark}[corollary]{Remark}

\newtheorem{problem}[corollary]{Problem}
\newtheorem{challenge}[corollary]{Challenge}

\begin{document}

\title{Bi-Lipschitz Bijection between the Boolean Cube \\ and the Hamming Ball}

\newcommand*\samethanks[1][\value{footnote}]{\footnotemark[#1]}
\author{Itai Benjamini\thanks{
    Faculty of Mathematics and Computer Science,
    Weizmann Institute of Science, Rehovot, {\sc Israel}.
    {\tt\{itai.benjamini,gil.cohen,igor.shinkar\}@weizmann.ac.il.}}
  \and
  Gil Cohen\samethanks
  \and
  Igor Shinkar\samethanks
}


\maketitle

\begin{abstract}
We construct a bi-Lipschitz bijection from the Boolean cube to the Hamming ball of equal volume.
More precisely, we show that for all even $n \in \N$ there exists an explicit bijection
$\f \colon \{0,1\}^n \to \left\{ x \in \{0,1\}^{n+1} \colon |x| > n/2 \right\}$
such that for every $x \neq y \in \cc$ it holds that
\[
\frac{1}{5}\leq \frac{\dist(\f(x),\f(y))}{\dist(x,y)} \leq 4,
\]
where $\dist(\cdot,\cdot)$ denotes the Hamming distance.
In particular, this implies that the Hamming ball is bi-Lipschitz transitive.

\medskip

This result gives a strong negative answer to an open problem of Lovett and Viola~[CC 2012],
who raised the question in the context of sampling distributions in low-level complexity classes.
The conceptual implication is that the problem of proving lower bounds in the context of
sampling distributions will require some new ideas beyond the
sensitivity-based structural results of Boppana~[IPL 97].

\medskip

We study the mapping $\f$ further and show that it (and its inverse) are computable in DLOGTIME-uniform $\TC{0}$, but not in $\AC{0}$. Moreover, we prove that $\f$ is ``approximately local'' in the sense that all but the last output bit of $\f$ are essentially determined by a single input bit.


\end{abstract}

\thispagestyle{empty}
\pagebreak

\newpage
\thispagestyle{empty}
\small
\hypersetup{linkcolor=black}
\tableofcontents
\normalsize
\newpage
\hypersetup{linkcolor=blue}
\pagenumbering{arabic}

\section{Introduction}\label{sec:intro}

The Boolean cube $\cc$ and the Hamming ball $\bb = \{x \in \cc[n+1] \colon |x|>n/2\}$, equipped with the Hamming distance, are two fundamental combinatorial
structures that exhibit, in some aspects, different geometric properties.
As a simple illustrative example, for an even integer $n \in \N$, consider
the vertex and edge boundaries%
\footnote{The edge boundary of a subset $A\subset \cc[n+1]$ is set of
edges with one endpoint in $A$ and one outside $A$.
The vertex boundary of $A$ is the set of vertices outside $A$ that are endpoints of boundary edges.}
of $\cc$ and $\bb$, when viewed as subsets of $\cc[n+1]$ of equal density $1/2$.
The Boolean cube is easily seen to maximize vertex boundary among all subsets of equal density (since all its vertices lie on the boundary), whereas Harper's vertex-isoperimetric inequality~\cite{Harper66} implies that the Hamming ball is in fact the unique minimizer. The same phenomena occurs for edge boundary, though interestingly, the roles are reversed:  among all monotone sets%
\footnote{Recall that a subset $A\subset \cc[n+1]$ is \emph{monotone} if $x \in A$ implies $y \in A$ for all $y\succeq x$.} of density $1/2$, the Poincar\'e inequality implies that the Boolean cube is the unique minimizer of edge boundary, whereas a classical result of Hart shows that the Hamming ball is the unique maximizer~\cite{Hart76}. From the Boolean functions perspective, the indicator of $\cc$ embedded in $\cc[n+1]$ is commonly referred to as the \emph{dictator} function, and the indicator of $\bb\subset \cc[n+1]$ is the \emph{majority} function, and it is a recurring theme in the analysis of Boolean functions that they are, in some senses, opposites of one another.

Lovett and Viola~\cite{LV11} suggested to utilize the opposite structure of the Boolean cube and the Hamming ball for proving lower bounds on sampling by low-level complexity classes such as $\AC{0}$ and $\TC{0}$. In particular, Lovett and Viola were interested in proving that for any even $n$, any bijection $f \colon \cc \to \bb$ has a large average stretch, where
\[
\avgstretch(f) = \Ex_{\substack{x \sim \cc \\ i \sim  [n]}} \left[ \dist(f(x),f(x+e_i)) \right],
\]
and $\dist(\cdot,\cdot)$ denotes the Hamming distance.
To be more precise, Lovett and Viola raised the following open problem.
\begin{problem}[{\cite{LV11}, Open Problem 4.1}]\label{prob:dict to maj}
Let $n \in \N$ be an even integer. Prove that for any bijection
$f \colon \cc \to \bb$, it holds that
\begin{equation}\label{eq:superpolylog stretch}
\avgstretch(f) = (\log{n})^{\omega(1)}.
\end{equation}
\end{problem}
\noindent
A positive answer to Problem \ref{prob:dict to maj} would demonstrate yet
another scenario in which the Boolean cube and the Hamming ball have a different
geometric structure -- any bijection from the former to the latter does not
respect distances. Furthermore, a positive answer to Problem~\ref{prob:dict to maj}
would have applications to lower bounds for sampling in $\AC{0}$; even a weaker
claim, where the right hand side in Equation~\eqref{eq:superpolylog stretch}
is replaced with $\omega(1)$, would have implications for sampling in the lower
class $\NC{0}$. We discuss this further in Section~\ref{sec:motivation}.

Arguably, the simplest and most natural bijection $\varphi \colon \cc \to \bb$ to consider is the following.
\begin{equation*}
\varphi(x) = \left\{
\begin{array}{ll}
 {\sf flip}(x) \circ 1 & \text{\quad if $|x| \le n/2$}  \\
 x \circ 0 & \text{\quad otherwise},
\end{array}
\right.
\end{equation*}
where ${\sf flip}(x)$ denotes the bit-wise complement of $x$.  It is straightforward to verify that $\avgstretch(\varphi) = \Theta(\sqrt{n})$. To see this, note that any edge $(x,y)$ in $\cc$, where $|x| = n/2$ and $|y| = n/2+1$, contributes $n$ to the average stretch, whereas all other edges contribute $1$. The assertion then follows since $\Theta(1/\sqrt{n})$ fraction of the edges are of the first type. In fact, the maximum stretch of $\varphi$ is $n$, where
\[
\maxstretch(\varphi) = \max_{\substack{x \in \cc \\ i \in  [n]}} \dist(\varphi(x),\varphi(x+e_i)).
\]
As far as we know, prior to our work this simple bijection achieved the best-known upper bound on the average stretch between $\cc$ and $\bb$, and no non-trivial upper bounds ({\it i.e.} sublinear) on maximum stretch were known.
For a survey on metric embeddings of finite spaces see~\cite{Li02}.
In particular, a lot of research has been done on the question of embedding
into the Boolean cube. For example, see~\cite{AB07,HLN87} for some work on
embeddings between random subsets of the Boolean cube,
and~\cite{G88} for isometric embeddings of arbitrary graphs into the Boolean cube.

\subsection{Our Results}
The main result of this paper is a strong \emph{negative} answer to
Problem~\ref{prob:dict to maj}.
\begin{theorem}[Main theorem]\label{thm:main}
For all even integers $n$, there exists a bijection $\f \colon \cc \to \bb$ with
\[
\maxstretch(\f) \le 4
\]
and
\[
\maxstretch(\f^{-1}) \le 5.
\]
\end{theorem}
\noindent
Theorem~\ref{thm:main} highlights a surprising geometric resemblance between
the Boolean cube and the Hamming ball.
In the language of metric geometry, Theorem~\ref{thm:main} says that there
is a bi-Lipschitz bijection between the two spaces.
\begin{corollary}[A bi-Lipschitz bijection between $\cc$ and $\bb$]
For all even integers $n$, there exists a bijection $\f\colon \cc \to \bb$, such that
for every $x \neq y \in \cc$ it holds that
\[
    \frac{1}{5}\leq \frac{\dist(\f(x),\f(y))}{\dist(x,y)} \leq 4.
\]
\end{corollary}
\noindent
As a corollary from Theorem~\ref{thm:main}, we obtain that the subgraph
of $\{0,1\}^{n+1}$ induced by the vertices of $\bb$ is \emph{bi-Lipschitz transitive}. Informally speaking, this says that any two points in $\bb$ have roughly the same ``view'' -- even the unique point with Hamming weight $n+1$
and the boundary points which have weight $n/2+1$. More formally,
\begin{corollary}[The Hamming balls are uniformly bi-Lipschitz transitive]
\label{cor:coarse transitivity}
For all even integers $n$, and for every two vertices $x,y \in \bb$ there is
a bijection $f \colon \bb \to \bb$ such that $f(x)=y$, $f(y) = x$, and
for every $z \neq w \in \bb$, it holds that
\[
    \frac{1}{20} \leq \frac{\dist(f(z),f(w))}{\dist(z,w)} \leq 20.
\]
\end{corollary}
\noindent
To see this, first note that $\bb$ is a convex subset of $\cc[n+1]$, and thus, the distances between vertices in $\bb$ are the same as their distances as a subset of the cube. Now, for a given pair $x,y \in \bb$, let $x' = \f^{-1}(x)$ and $y' = \f^{-1}(y)$, where $\f$ is the function from Theorem~\ref{thm:main}. Define $f \colon \bb \to \bb$ as $f(z) = \f( \f^{-1}(z) \oplus x' \oplus y')$. It is easy to see
that $f$ indeed satisfies the requirements of
Corollary~\ref{cor:coarse transitivity}.

\paragraph{Approximating $\boldsymbol{\f}_\mathbf{i}$}
We highlight another consequence of our main theorem that is perhaps somewhat surprising -- the bijection $\f$ of Theorem \ref{thm:main} is ``approximately local'' in the sense that almost \emph{all} of its output bits are essentially determined by only a \emph{constant} number of inputs bits.  To see this, we view the bijection $\f\colon\cc\to\bb$ as a vector of Boolean functions $\langle \f_1,\ldots, \f_{n+1}\rangle$, where $\f_i(x)$ is the $i^{\text{th}}$ output bit of $\f$ on input $x\in\cc$.  Recall that the \emph{total influence} of a Boolean function $\f_i \colon \cc \to\{0,1\}$ is the quantity
\[ \Inf[\f_i] =  \Ex_{x\sim \cc}\left[\#\{j\in [n] \colon \f_i(x) \ne \f_i(x+e_j)\}\right]. \]
By linearity of expectation, Theorem \ref{thm:main} implies that a typical $\f_i$ has bounded total influence.
\begin{equation} \Ex_{i \sim [n+1]} [\Inf[\f_i]] = \avgstretch(\f) \le \maxstretch(\f) \le 4. \label{small-total-influence}
\end{equation}
Next we recall Friedgut's Junta Theorem, which states that a Boolean function with bounded total influence is well-approximated by another Boolean function that only depends on a \emph{constant} number of input bits.  More precisely,
\newtheorem*{Friedgut}{Friedgut's Junta Theorem}
\begin{Friedgut}[\cite{F98}]
Let $f \colon \cc \to\{0,1\}$ be a Boolean function. For every
$\eps > 0$ there exists a Boolean function $g \colon \{0,1\}^n \to\{0,1\}$
such that $g$ is a $2^{O(\Inf[f]/\eps)}$-junta%
\footnote{Recall that a $k$-junta is a Boolean function that only depends on at most $k$ of its input bits.} and $\Pr[f(x) \ne g(x)] \le \eps$.
\end{Friedgut}
\noindent
Combining Equation~\eqref{small-total-influence} with Friedgut's Junta Theorem,
we see that for any constants $\delta,\eps > 0$, all but a $\delta$-fraction
of the $\f_i$'s are $\eps$-approximated by $O(1)$-juntas.
In fact, we give a direct proof that the function $\f$ in
Theorem~\ref{thm:main} satisfies the following stronger property.
\begin{proposition}\label{prop:f_i = x_i}
    For all $i \in [n]$ it holds that
    \[\Pr_{x}[\f_i(x) = x_i] > 1 - O(1/\sqrt{n}).\]
\end{proposition}
\noindent
That is, all but the last output bit of $\f$ are essentially determined by a \emph{single} input bit.

\paragraph{The complexity of $\boldsymbol{\f}$.}
Since the original motivation for constructing $\f$ comes from efficient
sampling of distributions, Theorem~\ref{thm:main} is of larger interest
if the bijection $\f$ (and $\f^{-1}$) can be computed by low-level circuits.

\begin{proposition}\label{prop:TC0}
    The bijections $\f$ and $\f^{-1}$ are computable in DLOGTIME-uniform $\TC{0}$.
\end{proposition}

\begin{remark}
    In fact, we show that for all $i \in [n+1]$ there is an $\NC{0}$-reduction from
    $\maj$ to $\f_i$. That is, $\TC{0}$ is the ``correct'' complexity of $\f$,
    and in particular, $\f$ is not in $\AC{0}$.
    See Proposition~\ref{prop:maj leq f_1} and Remark~\ref{rem:maj leq f_i}
    for details.
\end{remark}

\subsection{The Complexity of Distributions}\label{sec:motivation}

Lower bounds in circuit complexity are usually concerned with showing that a family of functions $\{f_n \colon \cc \to \{0,1\}\}_{n \in \N}$ cannot be computed by a family of circuits $\{C_n\}_{n \in \N}$ belonging to some natural class of circuits such as $\AC{0}$ or $\TC{0}$. Taking a broader interpretation of computation, it is often interesting to show that a class of circuits cannot perform a certain natural task beyond just computing a function.

One such natural task, introduced by Viola~\cite{Vi12}, is that of sampling distributions. In this problem, for a given distribution $\D$ supported
on $\cc$, we are looking for a function $f \colon \cc[m] \to \cc$
that samples (or approximates) $\D$, that is, for a uniformly random $x\!~\sim~\!\cc[m]$,
the distribution $f(x)$ is equal (or close to) $\D$, and furthermore, each output bit $f_i$
of the function $f$ belongs to some low-level complexity class, such as
$\AC{0}$ or $\TC{0}$.

As a concrete example, let $U_{\oplus}$ be the uniform distribution over the
set $\{ (x,\parity(x)) \colon x \in \cc[n-1] \} \seq \cc$.
Note that although the $\parity$ function is not computable in $\AC{0}$ (see~\cite{H86} and references therein) there is a function $f \colon \cc[n-1] \to \cc$ that samples $U_{\oplus}$,
such that each output bit depends on only two input bits:
\[
    f(x_1,\dots,x_{n-1}) = (x_1, x_1 + x_2, x_2 + x_3,\dots, x_{n-2} + x_{n-1},x_{n-1}).
\]
Motivated by the foregoing somewhat surprising example, Viola~\cite{Vi12}
raised the following question.
\begin{challenge}\label{challenge:main}
    Exhibit an explicit map $b \colon \cc \to \{0,1\}$ such that for a uniformly sampled
    $x \sim \cc$, the distribution $(x,b(x)) \in \cc[n+1]$ cannot be
    generated by $\poly(n)$-size $\AC{0}$ circuits given random bits as input.
\end{challenge}
\noindent
Note that the challenge does not specify the input length. Obviously, the
input length should be at least $n$, since this is the entropy of $(x,b(x))$,
and the challenge becomes harder as the input length increases. In this
paper we only consider distributions $\D$ that are uniform over subsets
$A \seq \cc$ of size $|A| = 2^m$ for some $m \in \N$, and functions with
input length $m$ sampling $\D$ (rather than approximating $\D$). Therefore, in this setting, the question of sampling a distribution $\D$ reduces to finding a bijection from $\cc[m]$
to $A$ such that each bit of the bijection is efficiently computable.

As a concrete candidate for the map $b \colon \cc \to \{0,1\}$ in
Challenge~\ref{challenge:main}, Viola suggested to use $\maj$ -- the other notoriously hard function for $\AC{0}$ other than $\parity$.
The following two problems have been stated in~\cite{LV11}.
\begin{problem}\label{problem:embedding into majset}
    Let $n \in \N$ be even. Does there exist a bijection
    $g \colon \cc \to \bb$ such that
    each output bit of $g$ is computable in $\AC{0}$?
\end{problem}

\begin{problem}\label{problem:embedding into (x,maj(x))}
    Let $n \in \N$ be odd. Does there exist a bijection
    $h \colon \cc \to \{ (x,\maj(x)) \colon x \in \cc \}$ such that
    each output bit of $h$ is computable in $\AC{0}$?
\end{problem}
\noindent
Note that a positive answer to Problem~\ref{problem:embedding into majset} implies
a positive answer to Problem~\ref{problem:embedding into (x,maj(x))}.
Indeed, if $g \colon \cc \to \bb$ is an embedding from
Problem~\ref{problem:embedding into majset}, then the function
$h \colon \cc[n+1] \to \cc[n+2]$ defined as
\[
    h(x_1, \dots, x_n,x_{n+1}) =
                    \begin{cases}
                        g(x_1,\dots,x_n) \circ 1 & \text{\quad $x_{n+1} = 1$} \\
                        {\sf flip}(g(x_1,\dots,x_n)) \circ 0 & \text{\quad $x_{n+1} = 0$} \\
                    \end{cases}
\]
gives an embedding for Problem~\ref{problem:embedding into (x,maj(x))}.
Therefore, a negative answer to Problem~\ref{problem:embedding into (x,maj(x))}
implies a negative answer to Problem~\ref{problem:embedding into majset}. In the other direction, if a function
$h \colon \cc \to \{ (x,\maj(x)) \colon x \in \cc \}$
gives a positive answer to Problem~\ref{problem:embedding into (x,maj(x))},
then the function $g \colon \cc \to \cc$ defined as%
\footnote{We use the following notation: for a string $s \in \cc$ and integers
$i \le j$ in $[n]$, the string $s_{[i,\ldots,j]}$ denotes the substring $s_i s_{i+1} \cdots s_j$.}
\[
    g(x_1, \dots, x_n) =
                    \begin{cases}
                        (h(x_1,\dots,x_n))_{[1,\dots,n]}            & \text{\quad} (h(x_1,\dots,x_n))_{n+1} = 1 \\
                        {\sf flip}(h(x_1,\dots,x_n))_{[1,\dots,n]}  & \text{\quad otherwise} \\
                    \end{cases}
\]
samples $\bb[n-1]$ using input of length $n$, which almost%
\footnote{Problem~\ref{problem:embedding into (x,maj(x))} asks for a function
that takes $n-1$ bits as input.}
answers Problem~\ref{problem:embedding into (x,maj(x))}.

On the positive side, Viola~\cite{Vi12} showed an explicit $\AC{0}$ circuit
$C \colon \{0,1\}^{\poly(n)} \to \{0,1\}^n$ of size $\poly(n)$
whose output distribution has statistical distance $2^{-n}$ from the uniform
distribution on $\{ (x,\maj(x)) \colon x \in \cc \}$.

Problem \ref{prob:dict to maj} was raised by Lovett and Viola~\cite{LV11} in an attempt to prove a lower bound for Problem~\ref{problem:embedding into majset}. A positive answer to Problem~\ref{prob:dict to maj} would imply a lower bound
for Problem~\ref{problem:embedding into majset}, since by the result
of~\cite{Bop}, any function $f \colon \cc[n] \to \cc[n+1]$ computable by a
polynomial size Boolean circuit of constant depth has average stretch at most
$\log^{O(1)}(n)$.

As we resolved Problem \ref{prob:dict to maj} negatively in a strong sense, it
seems that new ideas beyond the sensitivity-based structural results
of~\cite{Bop} will be required in order to resolve
Problems~\ref{problem:embedding into majset}
and~\ref{problem:embedding into (x,maj(x))}.

\subsection{Proof Overview}\label{sec:proof overview for main theorem}

In this section we describe in high-level the proof of Theorem~\ref{thm:main}. A full proof is given in Section~\ref{sec:proof of main theorem}. Let $n \in \N$ be an even integer. Our goal is to map $\cc$ to $\bb$ in a way that the two endpoints of every edge in $\cc$ are mapped to close vertices in $\bb$. The key building block
we use is a classical partition of the vertices of $\cc$ to symmetric chains, due to
De~Bruijn, Tengbergen, and Kruyswijk~\cite{BTK51},
where a symmetric chain is a path $\{c_k, c_{k+1}, \ldots, c_{n-k}\}$ in $\cc$, such that
each $c_i$ has Hamming weight $i$ (see Figure~\ref{fig:the partition}).

As a first step, we study the chains in the partition of De~Bruijn~\etal~ Roughly speaking, we show%
\footnote{This is somewhat implicit in our proofs, and is mentioned here mainly in order to build an intuition.}
that adjacent chains move closely to each other.
More precisely, if two adjacent vertices $x$ and $y$ belong to different chains,
then $x$ and $y$ have the same distance from the top of their respective chains, up to some additive constant. Moreover, the lengths of the two chains differ by at most some additive constant, and the $i^{\text{th}}$ vertex in one chain, when counting from the top, is $O(1)$-close to the $i^{\text{th}}$ vertex in the other chain (if such exists).

We now describe how to map $\cc$ to $\bb$ based on the partition of De~Bruijn~\etal~Consider a chain $c_k, c_{k+1}, \ldots, c_{n-k}$. Our mapping will ``squeeze'' the vertices to the top half of the cube while exploiting the extra dimension. In particular, every vertex will climb up its chain half the distance it has from the top, and then, the collision between two vertices is resolved by setting the extra last bit to 1 for the first vertex and to 0 for the second vertex. More precisely, the vertex $c_{n-k}$, which is at the top of its chain, is mapped to $c_{n-k} \circ 1$, while $c_{n-k-1}$ is mapped to $c_{n-k} \circ 0$. The third vertex from the top $c_{n-k-2}$ is mapped to $c_{n-k-1} \circ 1$ while $c_{n-k-3}$ is mapped to $c_{n-k-1} \circ 0$ and so on. The vertex $c_k$ at the bottom of the chain is mapped to $c_{n/2} \circ 1$, which is indeed in $\bb$.

Consider now two adjacent vertices $x,y$ in $\cc$. By the above, these vertices reside in ``close'' chains with roughly the same length and have roughly the same distance from the top of their respective chains. Thus, in the climbing process, both $x$ and $y$ will be mapped to vertices that have roughly the same distance from the top of their respective chains, and hence, from the discussion above, their images will be $O(1)$-close.

\section{Proof of the Main Theorem}\label{sec:proof of main theorem}

In this section we prove the main theorem. In Section~\ref{sec:the partition} we describe the De~Bruijn-Tengbergen-Kruyswijk partition. In Section~\ref{sec:f} we define the mapping $\f$ and prove basic facts about it. In Section~\ref{sec:finally the proof} we give the proof for Theorem~\ref{thm:main}, omitting some technical details that can be found in Section~\ref{sec:missing claims}.

\subsection{The De~Bruijn-Tengbergen-Kruyswijk Partition}\label{sec:the partition}

\begin{definition}
Let $n$ be an even integer. A \emph{symmetric chain} in $\cc$ is a
sequence of vertices $C = \{ c_k, c_{k+1}, \dots, c_{n-k} \}$ such that
$|c_i| = i$ for $i=k,k+1,\ldots,n-k$, and $\dist(c_i,c_{i+1}) = 1$ for
$i=k,k+1,\ldots,n-k-1$. We say that a symmetric chain is \emph{monotone} if it satisfies the following property: if $c_{i-1}$ and $c_i$ differ in the $j^{\text{th}}$
coordinate, and $c_i$ and $c_{i+1}$ differ in the $(j')^{\text{th}}$ coordinate,
then $j < j'$.
\end{definition}
\noindent
We shall represent a monotone symmetric chain as follows.
Let $y \in \regcube$ be such that $m = |\{ i \colon y_i = \emptyslot\}|$
satisfies $m \equiv n \pmod 2$, and let $k = (n-m)/2$.
The monotone symmetric chain $C_y = \{ c_k, c_{k+1}, \ldots, c_{n-k} \}$
is specified by $y$ as follows. For $i = k,k+1,\ldots,n-k$, the string $c_i$ is
obtained by replacing the $m-(i-k)$ leftmost symbols $\emptyslot$ of $y$ by $0$
and the remaining $i-k$ symbols $\emptyslot$ by $1$. Note that $C_y$ is indeed a
monotone symmetric chain.

De~Bruijn, Tengbergen, and Kruyswijk~\cite{BTK51} suggested a recursive
algorithm that partitions $\cc$ to monotone symmetric chains. We will follow
the presentation of the algorithm described in~\cite{LW2001} (see Problem 6E
in Chapter 6). The algorithm gets as input a string $x \in \cc$, and computes
a string $y \in \regcube$ which encodes the monotone symmetric chain $C_y$
that contains $x$.

The algorithm is iterative. During the running of the algorithm, every coordinate
of $x$ is either marked or unmarked, where we denote a marked $0$ by $\hat{0}$
and a marked $1$ by $\hat{1}$.
In each step, the algorithm chooses a consecutive pair $10$, marks it
by $\hat{1}\hat{0}$, temporarily deletes it, and repeats the process. The algorithm halts once
there is no such pair, that is  the remaining string is of the form
$00\dots01\dots11$. We call this stage of the algorithm the \emph{marking stage},
and denote the marked string by $\markhat(x) \in \{0,1,\hat{0},\hat{1}\}^n$.
The string $y$ is then defined as follows: if the $i^{\text{th}}$ bit of $x$
was marked then $y_i = x_i$. Otherwise, $y_i = \emptyslot$.

For example, consider the string $x = 01100110$. At the first iteration, the
algorithm may mark the third and fourth bits to obtain $01\hat{1}\hat{0}0110$.
Then, the second and fifth bits are marked 0\^{1}\^{1}\^{0}\^{0}110. Lastly,
the rightmost two bits are marked, and we obtain the marked string
$\markhat(x) = 0\hat{1}\hat{1}\hat{0}\hat{0}1\hat{1}\hat{0}$.
Hence $y = \emptyslot 1100 \emptyslot 10$ and
$C_y = \{ \und{0}1100\und{0}10, \und{0}1100\und{1}10, \und{1}11100\und{1}10 \}$.

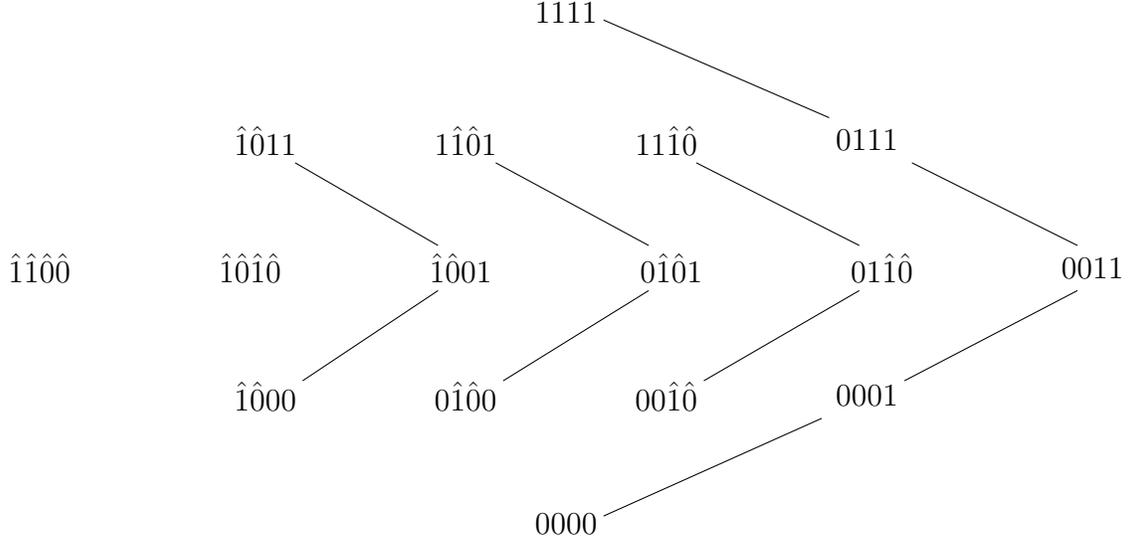
\begin{figure}\centering
\begin{tikzpicture}

\draw (-5,1.7) node {$\hat{1}\hat{0}00$};
\draw (-8+28/5,3.4) node {$\hat{1}\hat{0}01$};
\draw (-5,5.1) node {$\hat{1}\hat{0}11$};

\draw (-5+0.5,1.9) -- (-8+28/5-0.3,3.1);
\draw (-8+28/5-0.3,3.7) -- (-5+0.4,4.8);

\draw (-5+8/3,1.7) node {$0\hat{1}\hat{0}0$};
\draw (-8+42/5,3.4) node {$0\hat{1}\hat{0}1$};
\draw (-5+8/3,5.1) node {$1\hat{1}\hat{0}1$};

\draw (-5+8/3+0.5,1.9) -- (-8+42/5-0.3,3.1);
\draw (-8+42/5-0.3,3.7) -- (-5+8/3+0.4,4.8);

\draw (-5+16/3,1.7) node {$00\hat{1}\hat{0}$};
\draw (-8+56/5,3.4) node {$01\hat{1}\hat{0}$};
\draw (-5+16/3,5.1) node {$11\hat{1}\hat{0}$};

\draw (-5+16/3+0.5,1.9) -- (-8+56/5-0.3,3.1);
\draw (-8+56/5-0.3,3.7) -- (-5+16/3+0.4,4.8);

\draw (-1,0) node {$0000$};
\draw (3,1.7) node {$0001$};
\draw (6,3.4) node {$0011$};
\draw (3,5.1) node {$0111$};
\draw (-1,6.8) node {$1111$};

\draw (-0.5,0.1) -- (2.4,1.4);
\draw (3.5,1.9) -- (5.8,3.1);
\draw (5.8,3.7) -- (3.6,4.8);
\draw(2.5,5.4) -- (-0.5,6.7);

\draw (-8+14/5,3.4) node {$\hat{1}\hat{0}\hat{1}\hat{0}$};
\draw (-8,3.4) node {$\hat{1}\hat{1}\hat{0}\hat{0}$};

\end{tikzpicture}
\caption{The De~Bruijn-Tengbergen-Kruyswijk Partition for $n=4$.}
\label{fig:the partition}
\end{figure}

Readily, the algorithm induces a partition of $\cc$ to monotone symmetric chains.
We stress that although the algorithm has some degree of freedom when choosing
a $10$ pair out of, possibly, many pairs in a given iteration, the output
of the algorithm, $y$, is independent of the specific choices that were made.
That is, $y$ is a function of $x$, and does not depend on the specific order in which the algorithm performs the marking. This assertion can be
proven easily by induction on $n$. As a consequence, we may choose the
ordering of the $10$ pairs as we wish. We will use this fact in the proof
of Theorem~\ref{thm:main}.

\subsection{The Bijection $\boldsymbol{\f}$}\label{sec:f}

We define the mapping $\f$ as follows. Let $n \in \N$ be an even integer. For an input $x \in \cc$, let
$C = \{ c_k, c_{k+1}, \ldots, c_{n-k}\}$ be the symmetric chain from the partition of De~Bruijn~\etal, that contains $x$. Let $j$ be the index such that $x = c_j$.
Define
\begin{equation}\label{eq:definition of f}
    \f(x) \eqdef \begin{cases}
                    c_{\frac{(n-k)+j}{2}} \circ 1    & j \equiv (n-k) \pmod{2}; \\
                    c_{\frac{(n-k)+j+1}{2}} \circ 0    & j \not\equiv (n-k) \pmod{2}.
                \end{cases}
\end{equation}

\begin{claim}\label{claim:f is a bijection}
    The mapping $\f$ is a bijection from $\cc$ to $\bb$.
\end{claim}
\begin{proof}
We first show that the range of $\f$ is $\bb$. Consider $x \in \cc$ and let $C = \{ c_k, c_{k+1}, \ldots, c_{n-k}\}$ be the symmetric chain that contains $x$. Suppose that $x = c_j$ for some $k \leq j \leq n-k$.
If $j \equiv (n-k) \pmod 2$, then using the fact that $j \geq k$,
\[
|\f(x)| = \left| c_\frac{n-k+j}{2}  \circ 1 \right| = \frac{n-k+j}{2} + 1 > \frac{n}{2}.
\]
Otherwise, $j \not\equiv (n-k) \pmod 2$. Since $n$ is even, it follows that
$j \not \equiv k \pmod2$, and thus $j \geq k+1$. Hence,
\[
|\f(x)| = \left| c_\frac{n-k+j+1}{2} \circ 0 \right| > \frac{n}{2}.
\]
In both cases $\f(x) \in \bb$.

We conclude the proof by describing the inverse mapping $\f^{-1} \colon \bb \to \cc$. For $z \in \bb$, write $z = x \circ z_{n+1}$,
where $x \in \cc$ and $z_{n+1}$ is the $(n+1)^{\text{st}}$ bit of $z$.
Let $C = \{ c_k, c_{k+1}, \ldots, c_{n-k}\}$ be the symmetric chain that contains $x$,
and let $j$ be the index such that $x = c_j$ (note that $j \ge n/2$). Then,
\begin{equation}\label{eq:f inverse}
    \f^{-1}(z) = \begin{cases}
                    c_{2j-(n-k)}    & \text{ if } z_{n+1} = 1; \\
                    c_{2j-(n-k)-1}    & \text{ if } z_{n+1} = 0.
                \end{cases}
\end{equation}
It is straightforward to verify that this is indeed the inverse mapping of $\f$.
\end{proof}
\noindent
In order to understand the mapping $\f$ better, consider $x \in \cc$ and let $y \in \regcube$ be the encoding of the chain that
contains $x$. Note that if $1 \le i_1 < i_2 < \cdots < i_t \le n$ are the
coordinates in which $y$ contains $\emptyslot$, then there exists some
$0 \le \ell \le t$ such that $x_{i_1} = \cdots = x_{i_\ell} = 0$ and
$x_{i_{\ell+1}} = \cdots = x_{i_t} = 1$.
That is, $x$ is located at the $(\ell+1)^{\text{st}}$ position
of the chain $C_y$, when counting from the top. The function $\f$ outputs
the vertex located at the $(\floor{\ell/2}+1)^{\text{st}}$ position in the chain,
concatenated with 1 or 0, depending on the parity of $\ell$.
In other words, we obtain $\f(x)$ by keeping intact all the bits of $x$
in the coordinates other than $i_{\floor{\ell/2}+1}, \ldots, i_\ell$,
and by setting $\f(x)_{i_{\floor{\ell/2}+1}} = \cdots = \f(x)_{i_{\ell}} = 1$.
Then, we append 1 to the obtained string if $\ell$ is even,
and append 0 otherwise.

For example, let us consider the Boolean cube $\cc[4]$, whose partition is
presented in Figure~\ref{fig:the partition}, and write explicitly where each
vertex in the chain of length 5 is mapped under $\f$.
\begin{align*}
&\f(1111) = 1111 \circ 1,\\
&\f(0111) = 1111 \circ 0,\\
&\f(0011) = 0111 \circ 1,\\
&\f(0001) = 0111 \circ 0,\\
&\f(0000) = 0011 \circ 1,
\end{align*}
where we write the concatenation mark $\circ$ only for the sake readability.

\medskip
The following claim is immediate from the definition of $\f$.
\begin{claim}\label{claim:mark to infl}
    Fix a string $x \in \cc$. Let $M \seq [n]$ be the set of marked coordinates in $\markhat(x)$.
    Then,
    \begin{itemize}
        \item For every $i \in  M$ it holds that $\f(x)_i = x_i$.
        \item For every $j \in [n] \setminus M$, the $j^{\text{th}}$ coordinate
        of $\f(x)$ does not depend on any of the bits $\{x_i\}_{i \in M}$.
    \end{itemize}
\end{claim}
\medskip
\noindent
We are now ready to prove Theorem~\ref{thm:main}.

\subsection{Proof of Theorem~\ref{thm:main}}\label{sec:finally the proof}

\begin{proof}[Proof of Theorem~\ref{thm:main}]
We first show that $\maxstretch(\f) \le 4$. Fix an edge in $\cc$, that is, take $x \in \cc$ and $i \in [n]$ such that $x_i = 0$. Our goal is to show that $\dist(\f(x), \f(x+e_i)) \leq 4$. Let
$y,y' \in \regcube$ be the encodings of the chains $C_y, C_{y'}$ that contain
$x,x+e_i$ respectively.
As mentioned in Section~\ref{sec:the partition}, the output of the algorithm
on input $x$ is independent of the order in which the algorithm marks the $10$ pairs.
Therefore, given an input $x$, we may perform the marking stage in three steps:
\begin{enumerate}
  \item Perform the marking stage on the prefix of $x$ of length $i-1$.
  \item Perform the marking stage on the suffix of $x$ of length $n-i$.
  \item Perform the marking stage on the resulting, partially marked, string.
\end{enumerate}
Since $x$ and $x+e_i$ agree on all but the $i^{\text{th}}$ coordinate, the
running of the marking stage in steps 1 and 2 yield the same marking.
That is, prior to the third step, the strings $x$ and $x+e_i$ have the same
bits marked. Denote by $s \in \markedcube[i-1]$ and $t \in \markedcube[n-i]$ the
two partially marked strings such that the resulted strings after the second
step on inputs $x$ and $x+e_i$ are $s \circ 0 \circ t$ and $s \circ 1 \circ t$
respectively.
Let us suppose for concreteness that the string $s$ contains $a$ unmarked
zeros and $b$ unmarked ones, and the string $t$ contains $c$ unmarked zeros
and $d$ unmarked ones. Recall that at the end of the marking stage, all
unmarked zeros are to the left of all unmarked ones in both $s$ and $t$.

By Claim~\ref{claim:mark to infl}, the only coordinates that may contribute to
$\dist(\f(x),\f(x+e_i))$ are the unmarked coordinates prior to the third step, and so
\[
    \dist(\f(x),\f(x+e_i)) = \dist(\f(0^a 1^b \circ 0 \circ 0^c 1^d),\f(0^a 1^b \circ 1 \circ 0^c 1^d)).%
    \footnote{Note that $\f$ on the right hand side is applied to inputs of length not necessarily $n$.
    However, for the sake of readability, we do not indicate the input length when applying $\f$.
    In other words, $\f$ is a shorthand for a family of functions $\{ \f_n \}_{n \in \N}$.}
\]
Therefore, it is enough to bound from above the right hand side by $4$. At this point, it is fairly easy to be convinced that the right hand side is bounded by \emph{some} constant. Proving that the constant is $4$ is done by a somewhat tedious case analysis, according to the relations between $a,b,c$ and $d$. We defer the proof of the following claim to Section~\ref{sec:missing claims}.
\begin{claim}\label{claim:case analysis1}
    For every $a,b,c,d \in \N$, we have
    \[\dist(\f(0^a 1^b \circ 0 \circ 0^c 1^d),\f(0^a 1^b \circ 1 \circ 0^c 1^d)) \leq 4.\]
\end{claim}
\noindent
This completes the proof for $\maxstretch(\f) \le 4$.
\medskip

We now prove that $\maxstretch(\f^{-1}) \le 5$, where we use
the description of $\f^{-1}$ given in Equation~\eqref{eq:f inverse}.
In order to bound $\maxstretch(\f^{-1})$, let us fix an edge
in $\bb$, that is, take $z \in \bb$ and $i \in [n+1]$ such that $z_i = 0$
and show that $\dist(\f^{-1}(z), \f^{-1}(z+e_i)) \leq 5$.
By the proof of Claim~\ref{claim:f is a bijection},
if $i=n+1$ then $\f^{-1}(z)$ and $\f^{-1}(z+e_i)$ are consecutive vertices
in some monotone symmetric chain, and thus
$\dist(\f^{-1}(z), \f^{-1}(z+e_i)) = 1$.

Therefore, we shall assume henceforth that $i \neq n+1$. Let $z = x \circ z_{n+1}$ and $z+e_i = (x+e_i) \circ z_{n+1}$
for some $x \in \cc$ and $z_{n+1} \in \{0,1\}$.
Let $y,y' \in \regcube$ be the encodings of the chains $C_y, C_{y'}$ that
contain $x,x+e_i$ respectively. Similarly to the proof for $\maxstretch(\f) \le 4$,
we perform the marking stage by first performing the marking stage on
the prefix of $x$ of length $i-1$, then perform the marking stage on the suffix
of $x$ of length $n-i$, and finally, perform the marking stage on the resulting,
partially marked string.
Denote by $s \in \markedcube[i-1]$ and $t \in \markedcube[n-i]$ the
two partially marked strings such that the resulted strings after the second
step on inputs $x$ and $x+e_i$ are $s \circ 0 \circ t$ and $s \circ 1 \circ t$
respectively. Suppose again for concreteness that the string $s$ contains $a$ unmarked
zeros and $b$ unmarked ones, and the string $t$ contains $c$ unmarked zeros
and $d$ unmarked ones.

By Claim~\ref{claim:mark to infl}, the only coordinates that may contribute to
$\dist(\f^{-1}(z),\f^{-1}(z+e_i))$ are the unmarked coordinates in $s$ and $t$,
and so
\[
    \dist(\f^{-1}(z),\f^{-1}(z+e_i))
    =
    \dist(\f^{-1}(0^a 1^b \circ 0 \circ 0^c 1^d \circ z_{n+1}),\f^{-1}(0^a 1^b \circ 1 \circ 0^c 1^d \circ z_{n+1})).
\]
Therefore, it is enough to upper bound the right hand side by $5$. We first note that $a+c+1 \le b+d$. To see this recall that $|z| > n/2$ and $0^a 1^b \circ 0 \circ 0^c 1^d$ was obtain from $x = z_1\ldots z_n$ (that is, $z$ without its last bit $z_{n+1}$) by deleting the same number of zeros and ones.
\begin{claim}\label{claim:case analysis2}
    For every $a,b,c,d \in \N$ such that $a+c+1 \leq b+d$,
    and for every $z_{n+1} \in \{0,1\}$
    it holds that
    \[
    \dist(\f^{-1}(0^a 1^b \circ 0 \circ 0^c 1^d \circ z_{n+1}),
    \f^{-1}(0^a 1^b \circ 1 \circ 0^c 1^d \circ z_{n+1})) \leq 5.
    \]
\end{claim}
\noindent
Therefore, by Claim~\ref{claim:case analysis2} we have $\maxstretch(\f^{-1}) \leq 5$.
This completes the proof of Theorem~\ref{thm:main}.
\end{proof}

\subsection{Proof of Missing Claims}\label{sec:missing claims}
We now return to the proofs of
Claim~\ref{claim:case analysis1} and Claim~\ref{claim:case analysis2}.

\begin{proof}[Proof of Claim~\ref{claim:case analysis1}]
Let $w = 0^a 1^b \circ 0 \circ 0^c 1^d$
and $w' = 0^a 1^b \circ 1 \circ 0^c 1^d$.
We prove the claim using the following case analysis.
It will be convenient to introduce
the function $\even \colon \N \to \{0,1\}$ defined as
$\even(n) = 1$ if $n$ is even, and $\even(n) = 0$ otherwise.
\paragraph{Case 1 ($\boldsymbol{b = c}$).}
In this case we have
$w = 0^a \circ 1^{b} 0^{b} \circ 0 1^d$
and
$w' = 0^a 1 \circ 1^{b} 0^{b} \circ 1^d$.
After the marking stage we get
$\markhat(w) = \und{0}^a \circ \hat{1}^{b} \hat{0}^{b} \circ \und{0} \und{1}^d$
and
$\markhat(w') = \und{0}^a \und{1} \circ \hat{1}^{b} \hat{0}^{b} \circ \und{1}^d$.
Therefore,
\[
    \f(w) = 0^{\floor{\frac{a+1}{2}}} 1^{a-\floor{\frac{a+1}{2}}} \circ 1^{b} 0^{b} \circ 1^{d+1} \circ \even(a+1)
\]
and
\[
    \f(w') = 0^{\floor{\frac{a}{2}}} 1^{\ceil{\frac{a}{2}}+1} \circ 1^{b} 0^{b} \circ 1^d \circ \even(a).
\]
By inspection, one can now easily verify that $\dist(\f(w),\f(w')) \le 4$ in this case.

\paragraph{Case 2 ($\boldsymbol{b > c}$).}
In this case we have
$w = 0^a \circ 1^{b-c-1} \circ 1^{c+1} 0^{c+1} \circ 1^d$
and
$w' = 0^a \circ 1^{b-c+1} \circ 1^{c} 0^{c} \circ 1^d$.
After the marking stage we get
$\markhat(w) = \und{0}^a \und{1}^{b-c-1} \circ \hat{1}^{c+1} \hat{0}^{c+1} \circ \und{1}^d$
and
$\markhat(w') = \und{0}^a \und{1}^{b-c+1} \circ \hat{1}^{c} \hat{0}^{c} \circ \und{1}^d$.
Therefore,
\[
    \f(w) = 0^{\floor{\frac{a}{2}}} 1^{\ceil{\frac{a}{2}} + b-c-1} \circ 1^{c+1} 0^{c+1} \circ 1^d \circ \even(a)
\]
and
\[
    \f(w') = 0^{\floor{\frac{a}{2}}} 1^{\ceil{\frac{a}{2}} + b-c+1} \circ 1^{c} 0^{c} \circ 1^d \circ \even(a).
\]
Therefore, in this case, $\dist(\f(w),\f(w')) \leq 1$.

\paragraph{Case 3 ($\boldsymbol{b < c}$ and $\boldsymbol{a \geq c-b}$).}
In this case we have
$w = 0^a \circ 1^b 0^b \circ 0^{c-b+1} \circ 1^d$
and
$w' = 0^a \circ 1^{b+1} 1^{b+1} \circ 0^{c-b-1} \circ 1^d$.
After the marking stage we get
$\markhat(w) = \und{0}^a \circ \hat{1}^b \hat{0}^b \circ \und{0}^{c-b+1} \und{1}^d$
and
$\markhat(w') = \und{0}^a \circ \hat{1}^{b+1} \hat{0}^{b+1} \circ \und{0}^{c-b-1} \und{1}^d$.
By the assumption that $a \geq c-b$ we have $a \geq \floor{\frac{a+c-b+1}{2}}$,
and so
\[
    \f(w) = 0^{\floor{\frac{a+c-b+1}{2}}} 1^{a-\floor{\frac{a+c-b+1}{2}}}
    \circ 1^b 0^b \circ 1^{d+c-b+1} \circ \even(a+c-b+1)
\]
and
\[
    \f(w') = 0^{\floor{\frac{a+c-b-1}{2}}} 1^{a-\floor{\frac{a+c-b-1}{2}}}
    \circ 1^{b+1} 0^{b+1} \circ 1^{d+c-b-1} \circ \even(a+c-b-1).
\]
Therefore, by inspection we have $\dist(\f(w),\f(w')) \leq 4$ for this case.

\paragraph{Case 4 ($\boldsymbol{b < c}$ and $\boldsymbol{a < c-b}$).}
Just like in the previous case, we have $\markhat(w) = \und{0}^a \circ \hat{1}^b \hat{0}^b \circ \und{0}^{c-b+1} \und{1}^d$
and
$\markhat(w') = \und{0}^a \circ \hat{1}^{b+1} \hat{0}^{b+1} \circ \und{0}^{c-b-1} \und{1}^d$.
By the assumption that $a < c-b$, we have $a \leq \floor{\frac{a+c-b-1}{2}}$,
and so
\[
    \f(w) = 0^{a}
    \circ 1^b 0^b \circ 0^{\floor{\frac{a+c-b+1}{2}}-a} 1^{c-b+1+d-\floor{\frac{a+c-b+1}{2}}+a} \circ \even(a+c-b+1)
\]
and
\[
    \f(w') = 0^{a}
    \circ 1^{b+1} 0^{b+1} \circ 0^{\floor{\frac{a+c-b-1}{2}}-a} 1^{c-b-1+d-\floor{\frac{a+c-b-1}{2}}+a} \circ \even(a+c-b-1).
\]
Therefore, in this case, $\dist(\f(w),\f(w')) \le 2$.

This completes the proof of Claim~\ref{claim:case analysis1}.
\end{proof}
\noindent
We now turn to the proof of Claim~\ref{claim:case analysis2}.
\begin{proof}[Proof of Claim~\ref{claim:case analysis2}]
Let $w = 0^a 1^b \circ 0 \circ 0^c 1^d$ and $w' = 0^a 1^b \circ 1 \circ 0^c 1^d$.
Our goal is to show that $\dist(\f^{-1}(w \circ z_{n+1}),\f^{-1}(w' \circ z_{n+1})) \leq 5$.
Let us suppose for simplicity that $z_{n+1}=1$.
The case $z_{n+1}=0$ is handled similarly, and the same bound is achieved.
We prove the claim using the following case analysis.

\paragraph{Case 1 ($\boldsymbol{b = c}$).}
In this case we have
$w = 0^a \circ 1^{b} 0^{b} \circ 0 1^d$
and
$w' = 0^a 1 \circ 1^{b} 0^{b} \circ 1^d$.
After the marking stage we get
$\markhat(w) = \und{0}^a \circ \hat{1}^{b} \hat{0}^{b} \circ \und{0} \und{1}^d$
and
$\markhat(w') = \und{0}^a \und{1} \circ \hat{1}^{b} \hat{0}^{b} \circ \und{1}^d$.
The assumption $a+c+1 \leq b+d$ implies that in this case we have $d-a-1 \geq 0$.
Therefore,
\[
    \f^{-1}(w \circ 1) = 0^{a} \circ 1^{b} 0^{b} \circ 0^{a+2} 1^{d-a-1}
\]
and
\[
    \f^{-1}(w' \circ 1) = 0^{a+1} \circ 1^{b} 0^{b} \circ 0^{a-1} 1^{d-a+1}.
\]
Therefore,
$\dist(\f^{-1}(w \circ 1),\f^{-1}(w' \circ 1)) \leq 4$.

\paragraph{Case 2 ($\boldsymbol{b < c}$).}
In this case we have
$w = 0^a \circ 1^{b} 0^{b} \circ 0^{c-b+1} 1^d$
and
$w' = 0^a \circ 1^{b+1} 0^{b+1} \circ 0^{c-b-1} 1^d$.
After the marking stage we get
$\markhat(w) = \und{0}^a \circ \hat{1}^{b} \hat{0}^{b} \circ \und{0}^{c-b+1} \und{1}^d$
and
$\markhat(w') = \und{0}^a \circ \hat{1}^{b+1} \hat{0}^{b+1} \circ \und{0}^{c-b-1} \und{1}^d$.
Therefore,
\[
    \f^{-1}(w \circ 1) = 0^{a} \circ 1^{b} 0^{b} \circ 0^{a+2(c-b+1)} 1^{d-(a+c-b+1)}
\]
and
\[
    \f^{-1}(w' \circ 1) = 0^{a} \circ 1^{b+1} 0^{b+1} \circ 0^{a+2(c-b-1)} 1^{d-(a+c-b-1)}.
\]
Therefore,
$\dist(\f^{-1}(w \circ 1),\f^{-1}(w' \circ 1)) \le 3$.

\paragraph{Case 3 ($\boldsymbol{b > c}$).}
In this case,
$w = 0^a 1^{b-c-1} \circ 1^{c+1} 0^{c+1} \circ 1^d$
and
$w' = 0^a 1^{b-c+1} \circ 1^{c} 0^{c} \circ 1^d$.
After the marking stage we get
$\markhat(w) = \und{0}^a \und{1}^{b-c-1} \circ \hat{1}^{c+1} \hat{0}^{c+1} \circ \und{1}^d$
and
$\markhat(w') = \und{0}^a \und{1}^{b-c+1} \circ \hat{1}^{c} \hat{0}^{c} \circ \und{1}^d$.
\paragraph{Subcase 3.1 ($\boldsymbol{a < b-c}$)}
\[
    \f^{-1}(w \circ 1) = 0^{2a} 1^{b-c-a-1} \circ 1^{c+1} 0^{c+1} \circ 1^d
\]
and
\[
    \f^{-1}(w' \circ 1) = 0^{2a} 1^{b-c-a+1} \circ 1^{c} 0^{c} \circ 1^d.
\]
Thus, $\dist(\f^{-1}(w \circ 1),\f^{-1}(w' \circ 1)) \le 1$.

\paragraph{Subcase 3.2 ($\boldsymbol{a = b-c}$)}
\[
    \f^{-1}(w \circ 1) = 0^{2a-1} \circ 1^{c+1} 0^{c+1} \circ 0 1^{d-1}
\]
and
\[
    \f^{-1}(w' \circ 1) = 0^{2a} 1 \circ 1^{c} 0^{c} \circ 1^{d}.
\]
Therefore, in this case, $\dist(\f^{-1}(w \circ 1),\f^{-1}(w' \circ 1)) \le 3$.

\paragraph{Subcase 3.3 ($\boldsymbol{a > b-c}$)}
\[
    \f^{-1}(w \circ 1) = 0^{a+b-c-1} \circ 1^{c+1} 0^{c+1} \circ 0^{a-b+c+1} 1^{d-(a-b+c+1)}
\]
and
\[
    \f^{-1}(w' \circ 1) = 0^{a+b-c+1} \circ 1^{c} 0^{c} \circ 0^{a-b+c-1} 1^{d-(a-b+c-1)}.
\]
Thus, $\dist(\f^{-1}(w \circ 1),\f^{-1}(w' \circ 1)) \leq 5$.
%

This completes the proof of Claim~\ref{claim:case analysis2}.
\end{proof}

\section{The Mapping $\boldsymbol{\f}$ is Computable in DLOGTIME-uniform $\TC{0}$}\label{sec:complexity}
In this section we analyze the complexity of the bijection $\f$ described in the
proof of Theorem~\ref{thm:main}. We first claim that each output bit of $\f$ (and of $\f^{-1}$) can be computed in DLOGTIME-uniform $\TC{0}$. In Proposition~\ref{prop:maj leq f_1} and in the remark following it, we show that indeed $\TC{0}$ is the ``correct'' class for $\f$.
\paragraph{Proposition~\ref{prop:TC0} (restated).}
{\em
    The bijections $\f$ and $\f^{-1}$ are computable in DLOGTIME-uniform $\TC{0}$.
}

\medskip
\noindent
We prove the proposition only for $\f$. The proof for $\f^{-1}$ is very similar,
and we omit it.

\begin{proof}
We divide the proof into two steps. First we show that the marking stage can
be implemented in $\TC{0}$. Then, given the marking of an input, we show how
to compute $\f$ in $\TC{0}$. Both steps can be easily seen to be DLOGTIME-uniform.

Throughout the proof, the output
of the marking stage is represented by two bits for each coordinate, encoding
a symbol in $\{0,1,\hat{0}, \hat{1}\}$, where one bit represents the
Boolean symbol, and the other indicates whether the coordinate is marked or not.

\paragraph{Implementing the marking stage.}
Let $x \in \cc$. In order to implement the marking stage in $\TC{0}$, we observe that the $i^{\text{th}}$ coordinate in $x$ is marked if and only if there are coordinates
$s_i \leq i \leq e_i$ such that
\begin{enumerate}
    \item The number of ones in $x_{[s_i,\dots,e_i]}$ is equal to
    the number of zeros in $x_{[s_i,,\dots,e_i]}$.
    \item For every $k \in \{s_i,,\ldots,e_i\}$, the number of ones in
    the prefix $x_{[s_i,\dots,k]}$ is greater or equal to the number
    of zeros in $x_{[s_i,\dots,k]}$.
\end{enumerate}
Fix $i \in [n]$ and fix $s_i, e_i \in [n]$ such that $s_i \le i \le e_i$.
Thinking of the bit 1 as '(' and 0 as ')', the above two conditions are
equivalent to checking whether the string $x_{[s_i,\ldots,e_i]}$ of
parentheses is balanced, or in other words, deciding whether
$x_{[s_i,\ldots,e_i]}$ is in Dyck language. It is well-known that
Dyck language can be recognized in $\TC{0}$~\cite{L77}.
In fact, it is not hard to show that deciding whether a string of length $m$
is in Dyck language can be carried out by a DLOGTIME-uniform $\TC{0}$ circuit with size $O(m)$.

Now, for each $i \in [n]$, we go over all choices for $s_i,e_i$ in parallel, and take the $\OR$ of the $O(n^2)$ results. Thus, for each $i \in [n]$, there is a DLOGTIME-uniform $\TC{0}$ circuit with size $O(n^3)$ that decides whether the $i^{\text{th}}$ coordinate is marked or not.

\paragraph{Computing $\boldsymbol{\f}(x)$ from $\markhat(x)$:}

In order to compute $\f(x)$, let $\markhat(x) \in \markedcube$
be the marking of $x$. Since every marked coordinate will remain unchanged,
we need to consider only of the unmarked coordinates. Recall also that the
unmarked bits form a sequence of zeros followed by a sequence of ones.
That is, if we ignore the marked coordinates, then we get a string of the form
$0^a1^b$ for some $a=a(x), b=b(x)$, and the output should be
$0^{\floor{\frac{a}{2}}}1^{\ceil{\frac{a}{2}}+b} \circ \even(a)$ (recall that $\even(a) = 1$ if $a$ is even, and $\even(a) = 0$ otherwise).
This can be implemented as follows.
\begin{enumerate}
  \item Let $a$ be the number of unmarked zeros in $\markhat(x)$.
  \item For each $i \in [n]$, let $u_i = u_i(x)$ be the number of unmarked
  coordinates among $\{1,\dots,i\}$.
  \item For all unmarked coordinates $i\in [n]$, if $2u_i<a$, then set the $i^{\text{th}}$ bit of the output to 0. Otherwise, set the $i^{\text{th}}$ bit to $1$.
  \item
  Set the $(n+1)^{\text{st}}$ bit of the output to $\even(a)$.

\end{enumerate}
\noindent
It is easy to verify that given $\markhat(x)$, checking whether the inequality $2u_i<a$ holds
can be done in $\TC{0}$, and so the entire second step can be carried out by a $\TC{0}$ circuit.
\end{proof}
\noindent
We remark that the bijection $\f$ cannot be computed in $\AC{0}$.
For example, we prove that the first output bit of $\f$ cannot be computed in $\AC{0}$.

\begin{proposition}\label{prop:maj leq f_1}
    The function $\maj$ is $\NC{0}$-reducible to $\f_1$,
    i.e., $\maj \leq_{\NC{0}} \f_1$.
    In particular  $\f_1 \notin \AC{0}$.
\end{proposition}
\begin{proof}
We first note that $\f_1(x)=0$ if and
only if $x_1 = 0$ and $\markhat(x)$ contains at least two unmarked zeros.
For odd $n$, we construct a reduction $r \colon \cc \to \cc[3n+1]$ that
for input $x \in \cc$
outputs a string $r(x) \in \cc[3n+1]$ as follows.
Let $x' \in \cc[2n]$ be the string obtained from $x$ by replacing each
0 of $x$ with $10$, and by replacing each 1 of $x$ with $00$.
Define $r(x) = 0 \circ 1^{n}\circ x'$.
For example, if $x = 01101$, then $x' = 10 \circ 00 \circ 00 \circ 10 \circ 00$,
and $r(x) = 0 \circ 1^{5} \circ 10 \circ 00 \circ 00 \circ 10 \circ 00$.
By the definition of $r$, it is clear that each bit of $r(x)$ depends on at most one bit of $x$.
It is straightforward to check that $\maj(x) = \f_1(r(x))$, and the assertion, then, follows.
\end{proof}

\begin{remark}\label{rem:maj leq f_i}
    Note that the redution above also gives $\maj \leq_{\NC{0}} \f_{n+1}$.
    A similar proof also shows that $\maj \leq_{\NC{0}} \f_i$
    for all $i \in [n+1]$.
\end{remark}

\section{All But the Last Output Bit Depend Essentially on a Single Input Bit}\label{sec:f_i = n_i}

In this section we prove Proposition~\ref{prop:f_i = x_i}. We recall it
here for convenience.
\paragraph{Proposition~\ref{prop:f_i = x_i} (restated).}
{\em
    For all $i \in [n]$ it holds that
    \[
        \Pr_{x}[\f_i(x) = x_i] > 1 - O(1/\sqrt{n}).
    \]
}

\noindent
Before proving the proposition, we need to further study the structure of
the De~Bruijn-Tengbergen-Kruyswijk partition described in Section~\ref{sec:the partition}.
We start with the following claim.

\begin{claim}
    Let $n$ be an integer, and let $\P$ be a partition of $\cc$
    into symmetric chains. For every $1 \leq t \leq n+1$,
    let $M_t$ be the number of symmetric chains of length $t$ in $\P$.
    Then,
\[
    M_t = \left\{
      \begin{array}{ll}
        {n \choose \frac{n-t+1}{2}} - {n \choose \frac{n-t-1}{2}} & \hbox{\quad $t \not\equiv n \pmod 2$;} \\
        0 & \hbox{\quad otherwise.}
      \end{array}
    \right.
  \]
\end{claim}

\begin{proof}
    Note first that if $C = \{ c_k, c_{k+1}, \ldots, c_{n-k} \}$
    is a symmetric chain, then its length is $n-2k+1$.
    In particular, this implies that there are no symmetric chains
    of length $t$ where $t \equiv n \pmod 2$, and hence $M_t=0$
    for such $t$.

    Next, we prove the claim for $t \not\equiv n \pmod 2$.
    This is done by backward induction on $t$.
    For $t=n+1$ we clearly have a unique
    symmetric chain starting at $0^n$ and ending at $1^n$,
    and hence $M_{n+1} = 1$, as claimed.

    Before actually doing the induction step, let us consider the next case, namely,
    $t = n-1$. Note that only one of the vertices of Hamming weight $1$
    is contained in the unique chain of length $n+1$, and so, since
    distinct vertices with equal weight are contained in distinct
    symmetric chains, there are $n-1$ chains with bottom vertex of Hamming weight $1$. Therefore $M_{n-1} = n-1$, as claimed.

    For the general induction step, suppose that the claim holds for all
    $t'$ larger than $t$. We prove the assertion for $t \not\equiv n \pmod 2$.
    Every symmetric chain of length $t$ must be of the form
    $C = \{ c_k, c_{k+1}, \ldots, c_{n-k} \}$, where $k = \frac{n-t+1}{2}$.
    Since the chains of length greater than $t$ are disjoint, and each
    contains a vertex with Hamming weight $k$, it follows that the number of vertices
    with Hamming weight $k$ that are contained in chains of length greater than $t$ is
    $\sum_{t'>t} M_{t'} = {n \choose k-1}$.
    The remaining ${n \choose k} - {n \choose k-1}$ vertices must be contained
    in chains of length $t$, and so, since
    distinct vertices of Hamming weight $k$ are contained in distinct
    symmetric chains, it follows that there are ${n \choose k} - {n \choose k-1}$
    chains of length $t$.
\end{proof}
\noindent
The following corollary is immediate from the observation that any
$x \in \cc$ such that $\markhat(x)$ contains exactly $a$ unmarked zeros and
$b$ unmarked ones is contained in a unique chain of length $a+b+1$ in the
De~Bruijn-Tengbergen-Kruyswijk partition.

\begin{corollary}\label{cor:counting 0^a1^b}
    Let $n,a,b \in \N$ such that $a+b \equiv n \pmod 2$,
    and $a+b \leq n$.
    Then,
    \begin{enumerate}
    \item
        The number of $x \in \cc$ such that $\markhat(x)$ contains exactly
        $a$ unmarked zeros and $b$ unmarked ones is
        ${n \choose \frac{n-a-b}{2}} - {n \choose \frac{n-a-b-2}{2}}$.
    \item
        The number of $x \in \cc$ such that $\markhat(x)$ contains
        exactly $a$ unmarked zeros (and any number of unmarked ones)
        is ${n \choose \floor{\frac{n-a}{2}}}$.
    \end{enumerate}
\end{corollary}
\noindent
We are now ready to prove Proposition~\ref{prop:f_i = x_i}.
\begin{proof}[Proof of Proposition~\ref{prop:f_i = x_i}]
    Let $x \in \cc$, and let $\markhat(x)$ be its marking.
    Suppose that the unmarked coordinates in $\markhat(x)$ are
    $i_1 < i_2 < \cdots < i_t$, and let $0 \le \ell \le t$ be such that
    $x_{i_1} = \cdots = x_{i_\ell} = 0$ and
    $x_{i_{\ell+1}} = \cdots = x_{i_t} = 1$.
    Note that $\f_i(x) \neq x_i$ if and only if the $i^{\text{th}}$ coordinate
    is unmarked in $\markhat(x)$ and $i = i_j$ for some $j \in \{\floor{\frac{\ell}{2}}+1, \dots, \ell\}$.

    As in the proof of Theorem~\ref{thm:main}, it will be convenient to perform the following partial marking of $x$.
    Let us first perform the marking
    stage on the prefix of $x$ of length $i-1$, and denote the resulting
    string by $s \in \markedcube[i-1]$. Then, perform the marking stage on
    the suffix of $x$ of length $n-i$ and denote the result string by
    $t \in \markedcube[n-i]$.
    Suppose for concreteness that the string $s$ contains $a$ unmarked
    zeros and $b$ unmarked ones, and the string $t$ contains $c$ unmarked zeros
    and $d$ unmarked ones. By the definition of $\f$ we have $\f_i(x) \neq x_i$
    if and only if $x_i = 0$, $b = 0$ and $a \ge c$.
    Therefore, since each bit of $x$ is chosen independently, the resulting partially marked strings $s,t$ and the bit $x_i$ are also independent, and so
    \[
        \Pr[\f_i(x) \neq x_i] = \Pr[x_i = 0] \cdot \Pr[b = 0,a \geq c]
        = \half \cdot \sum_{k=0}^{n-i} \sum_{j=k}^{i} \Pr[a = j,b=0] \Pr[c = k].
    \]
    By Corollary~\ref{cor:counting 0^a1^b}, for $j \not\equiv i \pmod 2$ we have
    \[
        \Pr[a = j,b=0] = \frac1{2^{i-1}} \cdot \left( {i-1 \choose \frac{i-j-1}{2}} - {i-1 \choose \frac{i-j-3}{2}}\right),
    \]
    and
    \[
        \Pr[c = k]
        =
        \frac1{2^{n-i}} \cdot {n-i \choose \floor{\frac{n-i-k}{2}}}.
    \]
    Therefore, for every $k \leq i$ we have
    \[
        \sum_{j=k}^{i} \Pr[a = j,b=0] =
        \frac1{2^{i-1}} \cdot    \!\!\!\! \sum_{\substack{k \leq j \leq i \\ j \not\equiv i\!\!\!\!\pmod 2}}
                \!\!\left({i-1 \choose \frac{i-j-1}{2}} - {i-1 \choose \frac{i-j-3}{2}}\right)
                = \frac1{2^{i-1}} \cdot {i-1 \choose \floor{\frac{i-k-1}{2}}},
    \]
    and so
    \[
        \Pr[\f_i(x) \neq x_i]
        \leq
        \frac{1}{2^{n+1}} \cdot \!\!\sum_{k=0}^{\min(i,n-i)}
                    {n-i \choose \floor{\frac{n-i-k}{2}}} {i-1 \choose \floor{\frac{i-k-1}{2}}}.
    \]
    Let us assume that $i \geq n/2$ (the case of $i < n/2$ is handled similarly).
    Then, using the fact that
    ${i-1 \choose \floor{\frac{i-k-1}{2}}} \leq O(2^i/\sqrt{i})$ for all $k$,
    we have
    \[
        \Pr[\f_i(x) \neq x_i] = O\left(\frac{1}{\sqrt{i}}\right) \cdot
            \frac{1}{2^{n-i}} \cdot \sum_{k=0}^{n-i} {n-i \choose \floor{\frac{n-i-k}{2}}}.
    \]
    By the identity
    \[
    \sum_{k=0}^{n-i} {n-i \choose \floor{\frac{n-i-k}{2}}} =
    \sum_{j=0}^{n-i} {n-i \choose j} = 2^{n-i},
    \]
    we get
    $\Pr[\f_i(x) \neq x_i] = O(1/\sqrt{i})$, and so, since we assumed that
    $i \geq n/2$ we get that $\Pr[\f_i(x) \neq x_i] = O(1/\sqrt{n})$, as required.
\end{proof}

\section{Concluding Remarks and Open Problems}

\subsubsection*{Bi-Lipschitz bijection between balanced halfspaces.}

Let $a_0, \ldots, a_n \in \R$. The halfspace determined by the $a_i$'s is the set of all points $(x_1, \ldots, x_n) \in \{-1,1\}^n$
such that $a_0 + a_1 x_1 + \cdots + a_n x_n \ge 0$.%
\footnote{The $\{-1,1\}^n$ representation of the Boolean cube is more natural in the context of halfspaces.}
A balanced halfspace is a halfspace with $a_0 = 0$. The Boolean cube $\{-1,1\}^n$ embedded in the natural way in $\{-1,1\}^{n+1}$ and
the Hamming ball $\{ x \in \{-1,1\}^{n+1} : x_1 + \cdots + x_{n+1} \ge 0 \}$ are two examples of balanced halfspaces. We showed a bi-Lipschitz bijection between them.
It is therefore natural to ask the following question.
\begin{problem}\label{prob:half spaces}
Is there a bi-Lipschitz bijection between
\emph{any} two balanced halfspaces? Or even a bijection with constant average
stretch from the Boolean cube $\{-1,1\}^n$ to any balanced halfspace in $\{-1,1\}^{n+1}$?
\end{problem}
\noindent
In functions terminology, the Boolean cube $\{-1,1\}^n$ embedded in $\{-1,1\}^{n+1}$ is indicated by the dictator function, while the Hamming ball is indicated by the majority function. Problem~\ref{prob:half spaces} refers more generally to linear threshold functions. One attempt at solving Problem~\ref{prob:half spaces} positively, would be to generalize the partition of De~Bruijn~\etal~to general halfspaces.

Besides being a natural problem, a positive solution to Problem~\ref{prob:half spaces} may have implications to fully polynomial approximation scheme for counting solutions of the $0$-$1$ knapsack problem~\cite{MS04}.

\medskip
\noindent
Another interseting problem, inspired by Corollary~\ref{cor:coarse transitivity},
is the following.
\begin{problem}\label{prob:half spaces coarsely transitive}
Is it true that any halfspace is bi-Lipschitz transitive?
\end{problem}

\subsubsection*{Tightness of the stretch from the Boolean cube to the Hamming ball.}
One may ask whether the constants $4$ and $5$ in Theorem~\ref{thm:main} are tight. By a slight
variation on the proof of Theorem~\ref{thm:main}, we can show that there exists a bijection $\h \colon \cc \to \bb$ with $\maxstretch(\h) \le 3$, improving on Theorem~\ref{thm:main} in this respect. However, the maximum stretch of $\h^{-1}$ is unbounded.

\begin{theorem}\label{thm:maxstretch 3}
For all even integers $n$, define the bijection $\h \colon \cc \to \bb$ as follows.
Let $x \in \cc$, and let $C = \{ c_k, c_{k+1}, \ldots, c_{n-k}\}$ be the symmetric chain from the partition of De Bruijn~\etal, that contains $x$. Let $j$ be the index such that $x = c_j$. Define,
\[
    \h(x) \eqdef \begin{cases}
                    c_{n-j} \circ 1    & j \le n/2; \\
                    c_{j} \circ 0    &  \text{otherwise}.
                \end{cases}
\]
Then, $\maxstretch(\h) = 3$ and $\avgstretch(\h^{-1}) = 2+o(1)$.
\end{theorem}
\noindent
The proof of Theorem~\ref{thm:maxstretch 3} is similar to the proof of Theorem~\ref{thm:main}, and thus we omit it.
One can easily see that any bijection $f \colon \cc \to \bb$ has maximum stretch at least $2$.
Indeed, let $y = f(x) \in \bb$ be a point with Hamming weight $n/2+1$.
Then $y$ has only $n/2$ neighbors in $\bb$, which cannot accommodate all $n$ neighbors of $x \in \cc$.
We do not know whether the stretch $3$ of $\h$ in Theorem~\ref{thm:maxstretch 3} is tight or not, and leave it as an open problem.
What is the smallest possible stretch of a bijection from $\bb$ to $\cc$? Are the constants $4$ and $5$ optimal if one considers only bi-Lipschitz bijections? Is the constant $20$ in Corollary~\ref{cor:coarse transitivity} optimal?

\subsubsection*{Lower bounds on average and maximum stretch}

\begin{problem}\label{problem:unbounded stretch}
Exhibit an explicit subset $A \subset \cc[n+1]$ of density $1/2$ such that
any bijection $f \colon \cc \to A$ has $\avgstretch(f) = \omega(1)$,
or prove that no such subset exists.
\end{problem}
\noindent
As a concrete candidate, we suggest to consider sets $A = \{x \colon f(x)=1 \}$,
where $f$ is a monotone noise-sensitive function
(e.g., {\sf Tribes}\footnote{We note that {\sf Tribes} has density close to $1/2$.} or {\sf Recursive-Majority-of-Three}).
A sufficiently strong positive answer to this question would imply
a lower bound for sampling the uniform distribution on $A$ by low-level
complexity classes.

For maximum stretch, we remark that it is easy to prove a lower bound of $\Omega(\sqrt{n})$.  Let
\begin{eqnarray*}
\mathcal{B}_0 &=& \{ x \in \cc[n+1] \colon |x| \le (n/2)-c\sqrt{n} \}, \\
\mathcal{B}_1 &=& \{ x \in \cc[n+1] \colon |x| \ge (n/2)+c\sqrt{n} \}
\end{eqnarray*}
be two diametrically opposed Hamming balls ($\mathcal{B}_0$ is centered at $0^{n+1}$ and $\mathcal{B}_1$ at $1^{n+1}$), where $c > 0$ is a universal constant chosen so that each set will have density $1/4$ in $\cc[n+1]$.  Let $A = \mathcal{B}_0 \cup \mathcal{B}_1$ be their disjoint union.
It is easy to see that any bijection $f \colon \cc[n] \to A$ must have $\maxstretch(f) = \Omega(\sqrt{n})$.
Indeed, since $f^{-1}(\mathcal{B}_0)$ and $f^{-1}(\mathcal{B}_1)$ partition $\cc[n]$,
we may pick an arbitrary edge $(x,y) \in \cc[n]$ with $x\in f^{-1}(\mathcal{B}_0)$
and $y \in f^{-1}(\mathcal{B}_1)$, and observe that $\dist(f(x),f(y)) \ge 2c\sqrt{n}$.
We leave the question of improving this lower bound, or showing that $O(\sqrt{n})$ is tight, as an open problem.

\subsubsection*{Bijections from the Gale-Shapley algorithm for the stable marriage problem}
Let $A,B$ be two subsets of $\cc[n+1]$ with density $1/2$. Consider the
Gale-Shapley algorithm for the stable marriage problem, where each vertex
$v \in A$ ranks all the vertices in $B$ according to their distance to $v$
(breaking ties according some rule).
What can be said about the average stretch of the bijection obtained from
this algorithm?
Two interesting settings are (1) $A = \cc, B=\bb$ and (2) $A,B$ are random subsets of $\cc$ of density $1/2$.
For related work in this direction see Holroyd~\cite{Hol}. Another natural bijection to consider, suggested to us by Avishay Tal, is the one induced by the Hungarian method for the assignment problem~\cite{K55}.

\section*{Acknowledgement}\label{sec:ack}
We thank Li-Yang Tan for introducing us Problem~\ref{prob:dict to maj},
and for helpful discussions.
We also thank Ehud Friedgut for suggesting to use the
De~Bruijn--Tengbergen--Kruyswijk partition, which turned
out to be the key step in the proof of Theorem~\ref{thm:main}. Itai Benjamini would also like to thank Microsoft Research New England,
where  this research was started.

\bibliographystyle{alpha}
\bibliography{bibliography}

\end{document}